\documentclass[reqnt]{amsart}
\usepackage{amsmath,amsthm,amssymb,amscd,float}
\usepackage{subcaption}
\usepackage[shortlabels]{enumitem}
\usepackage{youngtab}
\usepackage{tikz}
\usepackage{xcolor}

\newtheorem{theorem}{Theorem}[section]
\newtheorem{lemma}[theorem]{Lemma}

\theoremstyle{definition}
\newtheorem{definition}[theorem]{Definition}
\newtheorem{example}[theorem]{Example}

\numberwithin{equation}{section}
\allowdisplaybreaks
\begin{document}

%
%
%
%
%
%
%
%
%

\title[$J$-generalization of the Rogers-Ramanujan-Gordon identities]
 {$J$-generalization of the Rogers-Ramanujan-Gordon identities via commutative algebra}
 
 \author[Alapan Ghosh]{Alapan Ghosh}
 \address{Department of Mathematics, Indian Institute of Technology Guwahati, Assam, India, PIN- 781039}
 \email{alapan.ghosh@iitg.ac.in}
 
 \author[Rupam Barman]{Rupam Barman}
 \address{Department of Mathematics, Indian Institute of Technology Guwahati, Assam, India, PIN- 781039}
 \email{rupam@iitg.ac.in}

\date{April 23, 2026}



\begin{abstract} 
The Rogers-Ramanujan-Gordon identities generalize the classical partition identities discovered independently by L. J. Rogers and S. Ramanujan. In 2021, Afsharijoo provided a commutative algebra proof of the Rogers-Ramanujan-Gordon identities. Building on the Afsharijoo's approach, we present a commutative algebra proof of a broader family of identities introduced by Coulson \textit{et al.}, which includes the Rogers-Ramanujan-Gordon identities as a special case. In the proof, we relate the generating functions associated with these identities to the Hilbert-Poincar\'e series of suitably constructed graded algebras.
\end{abstract}

\maketitle
\section{Introduction} 
A partition of a positive integer $n$ is a finite sequence of non-increasing positive integers $\lambda=(\lambda_1, \lambda_2, \ldots, \lambda_s)$ such that $\lambda_1+\lambda_2+\cdots +\lambda_s=n$. The integers $\lambda_j$ are called the \emph{parts} of the partition $\lambda$. Let $p(n)$ denote the number of partitions of $n$, with the convention that $p(0):=1$. 
\par In 1748, Leonhard Euler \cite{Euler_1748} discovered a fundamental and elegant partition identity asserting that the number of partitions of an integer $n$ into odd parts is equal to the number of partitions of $n$ into distinct parts. Since then, numerous remarkable partition identities have been established (see, for example, \cite{Andrews_1998}). In 1894, Leonard James Rogers \cite{Rogers_1894} discovered certain partition identities that went largely unnoticed at that time. These identities were later rediscovered by Srinivasa Ramanujan, whose work brought them to the attention of the mathematical community; they are now known as the Rogers-Ramanujan identities:
\begin{align*}
1+\sum_{n=1}^{\infty}\frac{q^{n^2}}{(1-q)\cdots (1-q^n)}=\prod_{n=0}^{\infty}\frac{1}{(1-q^{5n+1})(1-q^{5n+4})},\\
1+\sum_{n=1}^{\infty}\frac{q^{n^2+n}}{(1-q)\cdots (1-q^n)}=\prod_{n=0}^{\infty}\frac{1}{(1-q^{5n+2})(1-q^{5n+3})}.
\end{align*}
The Rogers-Ramanujan identities have been related to a large number of different areas of mathematics. These identities admit elegant partition-theoretic interpretations. They were first recognized by MacMahon \cite{MacMahon} and Schur \cite{Schur} as identities for integer partitions. They have also been linked to algebraic geometry \cite{Afsharijoo_2026, Mourtada_2013, Mourtada_2025}, modular forms \cite{Ono_2008a, Ono_2008b}, transcendental number theory \cite{Richmond_1981},  group theory \cite{Fulman_2000}, statistical mechanics \cite{Andrews-Baxter_1984, Baxter_1981} and many more. The first Rogers-Ramanujan identity states that the number of partitions of an integer $n$ in which the difference between any two parts is at least $2$ is equal to the number of partitions of $n$ into parts congruent to $1$ or $4$ modulo $5$. The second Rogers-Ramanujan identity asserts that the number of partitions of $n$ in which each part exceeds $1$ and the difference between any two parts is at least $2$ is equal to the number of partitions of $n$ into parts congruent to $2$ or $3$ modulo $5$.
\par Gordon \cite{Gordon_1961} generalized the Rogers-Ramanujan identities and provided a combinatorial proof; these results are now known as the Rogers-Ramanujan-Gordon identities, Gordon's identities, or Gordon's theorem. Andrews  \cite{Andrews_1974} later gave an analytic formulation, which is commonly referred to as the Andrews-Gordon identities. In \cite{Ono}, Griffin \textit{et al.} developed a broad mathematical framework of Rogers-Ramanujan-type identities which gave new connections between Lie algebras and the theory of modular functions.
\par We now recall the Rogers-Ramanujan-Gordon identities. Let $r$ and $i$ be positive integers with $1\leq i\leq r$. Let $A_{r,i}(n)$ denote the number of partitions of $n$ into parts which are not congruent to $0$ or $\pm i$ modulo $2r+1$. Let $B_{r,i}(n)$ denote the number of partitions $(\lambda_1,\lambda_2,\ldots,\lambda_s)$ of $n$ satisfying the following conditions:
\begin{enumerate}[1.]
\item $\lambda_m-\lambda_{m+r-1}\geq2$ and
\item at most $i-1$ parts are equal to $1$.
\end{enumerate} 
\begin{theorem}[Rogers-Ramanujan-Gordon identities]\label{Thm1}
Let $r$ and $i$ be positive integers with $1\leq i\leq r$. Then 
$$A_{r,i}(n)=B_{r,i}(n), ~~ \text{for ~all} ~n\geq0.$$
\end{theorem}
The case $r=2$ of Theorem~\ref{Thm1} gives the classical Rogers-Ramanujan identities, while the case 
$r=1$ leads to the trivial identity $1=1$. Throughout this article, we assume $r\geq2$.
\par For $1\leq i\leq r$, define
\begin{align*}
\mathcal{A}_i(q):=\prod_{\substack{m\geq1,\\ m\not\equiv0,~\pm(r-i+1)\pmod{2r+1}}}\frac{1}{(1-q^m)}.
\end{align*}
Note that $\mathcal{A}_{r-i+1}(q)$ is the generating function for $A_{r,i}(n)$.
For $g\geq1$, the series $\mathcal{A}_{(r-1)g+i}(q)$ is defined recursively (see \cite[Section 2]{Lepowsky and Zhu 2012}) as follows: For $i=1$,
\begin{align*}
\mathcal{A}_{(r-1)g+1}(q)=\mathcal{A}_{(r-1)(g-1)+r}(q),
\end{align*}
and for $i=2,\ldots,r$,
\begin{align*}
\mathcal{A}_{(r-1)g+i}(q)=\frac{\mathcal{A}_{(r-1)(g-1)+r-i+1}(q)-\mathcal{A}_{(r-1)(g-1)+r-i+2}(q)}{q^{g(i-1)}}.
\end{align*}
Let $r$ and $i$ be positive integers with $1 \leq i \leq r$, and let $J \geq 0$ be an integer. Let $B_{r,i,J}(n)$ denote the number of partitions $(\lambda_1,\lambda_2,\ldots,\lambda_s)$ of $n$ satisfying the following conditions:
\begin{enumerate}[1.]
\item $\lambda_m-\lambda_{m+r-1}\geq2$,
\item all parts are greater than $J$, and
\item at most $i-1$ parts are equal to $J+1$.
\end{enumerate} 
Let $\mathcal{B}_{r,i,J}(q)$ denote the generating function of $B_{r,i,J}(n)$. The following theorem, due to Coulson \textit{et al.}~\cite{Coulson_2017}, establishes identities that generalize the Rogers-Ramanujan-Gordon identities in the sense that the first $J$ non-negative numbers do not appear as parts in the partition.
\begin{theorem}\cite[Theorem 8.8]{Coulson_2017}\label{Thm2}
For any nonnegative integer $J$ and $1\leq i\leq r$, we have
\begin{align*}
\mathcal{A}_{(r-1)J+\ell}(q)=\mathcal{B}_{r,i,J}(q),
\end{align*}
where $\ell=r-i+1$.
\end{theorem} 
Note that $B_{r,i,0}(n)=B_{r,i}(n)$. Thus, the case $J=0$ of Theorem \ref{Thm2} yields the Rogers-Ramanujan-Gordon identities. For $J \geq 1$, it is not known whether the series $\mathcal{A}_{(r-1)J+\ell}(q)$ admits a partition-theoretic interpretation, in contrast to the case $J = 0$, where $\mathcal{A}_{\ell}(q)$ is the generating function for the partition function $A_{r,i}(n)$, with $\ell = r - i + 1$.
\par 
The aim of this article is to present a commutative algebra proof of Theorem~\ref{Thm2}. In 2021, Afsharijoo \cite{Afsharijoo_2021} gave a commutative algebra proof of these identities corresponding to the $J=0$ case of Theorem~\ref{Thm2}. In this article, we extend Afsharijoo's approach to obtain a commutative algebra proof of Theorem~\ref{Thm2} in full generality.
\section{Preliminaries}\label{Section2}
In this section, we recall some definitions and results from commutative algebra and topology. For further details, see, for example, \cite{Atiyah_book, Eisenbud, Pfister_book}.
\begin{definition}[Graded ring]
A graded ring is a ring $A$ together with a family $(A_j)_{j \geq 0}$ of subgroups of the additive group of $A$, such that $A=\bigoplus_{j=0}^{\infty} A_j$ and $A_{j_1}A_{j_2} \subseteq A_{j_1+j_2}$ for all $j_1,j_2 \geq 0$.
\end{definition}
Here $A_0$ is a subring of $A$ and each $A_j$ is an $A_0$-module. For $j \geq 0$, $A_j$ is called the \emph{$j$-th homogeneous component} in the gradation of $A$. A nonzero element of $A_j$ is called a homogeneous element of degree $j$.
\begin{definition}[Homogeneous ideal]
An ideal $I$ of a graded ring $A$ is called homogeneous if it is generated by homogeneous elements.
\end{definition}
The intersection of a homogeneous ideal $I$ with $A_j$ is an $A_0$-submodule of $A_j$, called the homogeneous part of degree $j$ of $I$. A homogeneous ideal $I$ is the direct sum of its homogeneous parts $I_j=I \cap A_j$, i.e., $I=\bigoplus_{j=0}^{\infty} I_j$. If $I$ is a homogeneous ideal of a graded ring $A$, then the quotient ring $\frac{A}{I}$ is also a graded ring, decomposed as $$\frac{A}{I}=\bigoplus_{j=0}^{\infty} \frac{A_j}{I_j}.$$  
\begin{definition}[Graded $\mathbb{F}$-algebra]
Let $\mathbb{F}$ be a field. A graded ring $A=\bigoplus_{j=0}^{\infty} A_j$ is called a graded $\mathbb{F}$-algebra if it is also an $\mathbb{F}$-algebra, and $A_j$ is a vector space for all $j\geq0$ with $A_0=\mathbb{F}$.
\end{definition}
\begin{definition}[Weight of a polynomial]
The weight of the monomial $x_{i_1}^ {{\alpha}_1}\cdots x_{i_m}^{{\alpha}_m} \in \mathbb{F}[x_1,x_2,\ldots]$ is defined as $\sum_{k=1}^{m}i_k\alpha_k$. A polynomial $f(x) \in  \mathbb{F}[x_1,x_2,\ldots] $ is said to be a homogeneous polynomial of weight $a$ if every monomial of $f(x)$ has the same weight $a$.
\end{definition}
\begin{example}[Gradation by weight]\label{Example2}
Let $\mathbb{F}$ be a field of characteristic zero. Then $A:=\mathbb{F}[x_1,x_2,\ldots]$ is a graded algebra. $A$ is graded by weight, i.e., $A=\bigoplus_{j=0}^{\infty} A_j$, where $A_j$ is the set of polynomials of weight $j$ along with zero polynomial.
\end{example}
\begin{definition}[Hilbert-Poincar\'e series]
Let $\mathbb{F}$ be a field of characteristic zero and $A=\bigoplus_{j=0}^{\infty} A_j$ be a graded $\mathbb{F}$-algebra such that $\dim_{\mathbb{F}}(A_j)< \infty$. Then the Hilbert-Poincar\'e series of $A$ is $$\mathrm{HP}_A(q):=\sum_{j \geq 0}\dim_{\mathbb{F}}(A_j)q^j.$$
\end{definition} 
Next, we recall some facts about the Krull topology. For more details, see, for example \cite{Eisenbud}. Let $I$ be an ideal of ring $A$. The $I$-adic or Krull topology is a topology on $A$ in which a subset $U$ of $A$ is open if, for every $x \in U$, there exists $j \in \mathbb{N}$ such that $x+I^j \in U$. A sequence $(a_m)$ in $A$ converges to an element $a \in A$ if, for every $j \in \mathbb{N}$, there exists $N \in \mathbb{N}$ such that $(a_m-a) \in I^j$ for all $m \geq N$. In this article, we equip $A=\mathbb{F}[[q]]$ with 
the $I$-adic topology, where $I$ is an ideal generated by $q$ in $A$. We refer to this as the \emph{$q$-adic topology}.  
\section{A Proof of Theorem \ref{Thm2}}\label{Section3}
To prove the identities in Theorem \ref{Thm2} for $r \geq 2$, we first relate the generating function of $B_{r,i,J}(n)$ to the Hilbert-Poincar\'e series of a certain graded algebra. Let $\mathbb{F}$ be a field of characteristic zero, and consider the graded algebra $$S:=\mathbb{F}[x_1,x_2,x_3,\ldots],$$ graded by weight as in Example \ref{Example2}. For each $k \geq 1$, denote $$S_{k}:=\mathbb{F}[x_k,x_{k+1},x_{k+2},\ldots],$$ so that $S_1=S$. Let $(S_k)_j$ denote the homogeneous part of degree $j$ of the graded algebra $S_k$. Let $a$ and $t$ be integers. For $r\geq2$, $1\leq i\leq r$, and a fixed nonnegative integer $J$, we consider the ideal
\begin{align*}
&P_{r,i,J}:=\left(x_{J+1}^{i},~x_{J+1}^{i-1}x_{J+2}^{r-i+1},x_{J+1}^{i-2}x_{J+2}^{r-i+2},\ldots,x_{J+1}x_{J+2}^{r-1},~x_{a}^{r-t} x_{a+1}^{t} \right.\\
&\left.:a \geq J+2;~ 0 \leq t \leq r-1\right)
\end{align*}
of $S_{J+1}$, which is clearly a homogeneous ideal. Therefore, $\frac{S_{J+1}}{P_{r,i,J}}$ is also a graded algebra. For each $j\geq 0$, we have
\begin{align*}
\dim_{\mathbb{F}}\left(\frac{S_{J+1}}{P_{r,i,J}}\right)_j=\dim_{\mathbb{F}}\left(\frac{(S_{J+1})_j}{(P_{r,i,J})_j}\right)\leq\dim_{\mathbb{F}}((S_{J+1})_j) \leq \dim_{\mathbb{F}}((S)_j)= p(j)<\infty.
\end{align*}
Therefore, the Hilbert-Poincar\'e series of $\frac{S_{J+1}}{P_{r,i,J}}$ is well defined and given by
\begin{align*}
\mathrm{HP}_{\frac{S_{J+1}}{P_{r,i,J}}}(q)=\sum_{j\geq0}\dim_{\mathbb{F}}\left(\frac{S_{J+1}}{P_{r,i,J}}\right)_j q^j.
\end{align*}
We now relate the Hilbert-Poincar\'e series $\mathrm{HP}_{\frac{S_{J+1}}{P_{r,i,J}}}(q)$ to the partition function $B_{r,i,J}(n)$. We know that $P_{r,i,J}$ is generated by $x_{J+1}^{i}$, $x_{J+1}^{i-1}x_{J+2}^{r-i+1}, \ldots, x_{J+1}x_{J+2}^{r-1}$ and the monomials of the form  $x_{a}^{r-t} x_{a+1}^{t}$, such that $a \geq J+2$ and $0 \leq t \leq r-1$. Note that $\left(\frac{S_{J+1}}{P_{r,i,J}}\right)_j$ is generated by monomials, say $x_{l_1}x_{l_2} \cdots x_{l_m} \in S_{J+1}/P_{r,i,J}$, of weight $\sum_{p=1}^{m}l_p=j$. We associate a unique partition $(l_1,l_2,\ldots,l_m)$ of $j$ to this monomial which is counted by $B_{r,i,J}(j)$. Therefore, $\dim_{\mathbb{F}}\left(\frac{S_{J+1}}{P_{r,i,J}}\right)_j=B_{r,i,J}(j)$, and
\begin{align}\label{eqj1}
\mathrm{HP}_{\frac{S_{J+1}}{P_{r,i,J}}}(q)=\sum_{j\geq0}B_{r,i,J}(j)q^j=\mathcal{B}_{r,i,J}(q).
\end{align}
\par Let $a,~t$, and $\ell $ be integers such that $1 \leq \ell \leq r$. We define the following two ideals of $S_k$ for $k\geq J+1$:
\begin{equation*}
    P_{k}:=\left(x_{a}^{r-t} x_{a+1}^{t}:a \geq k,~0 \leq t \leq r-1\right)
\end{equation*}
and
\begin{equation*}
P_{k}^{\ell}:=
	\left(x_{k}^{\ell},~ x_{k}^{\ell-1}x_{k+1}^{r-\ell+1},~x_{k}^{\ell-2}x_{k+1}^{r-\ell+2},\ldots,x_{k}x_{k+1}^{r-1},~ P_{k+1} \right).
\end{equation*}
We note that $P_{r,i,J}=P_{J+1}^{i}$.
\par We denote the Hilbert-Poincar\'e series $\mathrm{HP}_{\frac{S_{k}}{P_{k}}}(q)$ by $\mathrm{HP}^{k}$ and the  Hilbert-Poincar\'e series $\mathrm{HP}_{\frac{S_{k}}{P_{k}^{\ell}}}(q)$ by $\mathrm{HP}_{\ell}^{k}$. Also, we use $\mathrm{HP}\left(\frac{A}{I}\right)$ in place of $\mathrm{HP}_{\frac{A}{I}}(q)$. With these notations, we note the following:
\begin{enumerate} 
\item[(N1)] $\mathrm{HP}_1^{k}= \mathrm{HP}^{k+1}$.
\item[(N2)] $\mathrm{HP}_r^{k}=\mathrm{HP}^{k}$.
\item[(N3)] $\mathrm{HP}\left(\frac{S_{J+1}}{P_{r,i,J}}\right)=\mathrm{HP}\left(\frac{S_{J+1}}{P_{J+1}^{i}}\right)=\mathrm{HP}_{i}^{J+1}$.
\end{enumerate}
Combining (N3) with \eqref{eqj1}, we obtain
\begin{align}\label{eqj2}
\mathrm{HP}_{i}^{J+1}=\mathcal{B}_{r,i,J}(q).
\end{align}
\par We now recall a result of Afsharijoo which gives a recursion formula for $\mathrm{HP}_{\ell}^{k}$ for all integers $k\geq J+1$.
\begin{lemma}[\cite{Afsharijoo_2021}]\label{Lemma2.1}
Let $J$ be a nonnegative integer. Let $k,~r$, and $\ell$ be positive integers with $r \geq 2$ and $1\leq \ell \leq r$. Then for any $k \geq J+1$, we have 
\begin{align}\label{n1}
\mathrm{HP}_\ell ^{k}=\sum_{j=1}^{\ell} q^{k(j-1)} \mathrm{HP}_{r-j+1} ^{k+1}. 
\end{align}
\end{lemma}
We rewrite \cite[Lemma 3.1]{Afsharijoo_2021} to obtain \eqref{n1}.
In the following lemma, we provide a recursion formula for $\mathrm{HP}_i ^{J+1}$.
\begin{lemma}\label{Lemma2.2}
Let $J$ be a nonnegative integer, and let $r$, $i$ be integers with $r \geq 2,~1 \leq i \leq r$. Then, for every integer $d \geq J+1$,  we have the following recursion formula:
\begin{align}\label{t1}
\mathrm{HP}_i ^{J+1}= \sum_{j=1}^{r} B_{i,j,(r-1)d+j}^{J} \mathrm{HP}_{r-j+1} ^{d+1}.
\end{align}
Here, the coefficients $B_{i,j,(r-1)d+j}^{J} \in \mathbb{F}[[q]]$ satisfy the following recursion formula for $1 \leq j \leq r$
$$B_{i,j,(r-1)(d+1)+j}^{J}= q^{(d+1)(j-1)} \sum_{m=1}^{r-j+1} B_{i,m,(r-1)d+m}^{J}$$
with the following initial conditions
$$B_{i,j,(r-1)(J+1)+j}^{J}=\begin{cases}
q^{(J+1)(j-1)} &\text{ if } 1 \leq j \leq i;  \\
0 &\text{ if } i+1 \leq j \leq r.
\end{cases}$$
\end{lemma}
\begin{proof}
To prove the required recursion formula \eqref{t1}, we use induction on $d$. First, we prove the formula for $d=J+1$. By using Lemma \ref{Lemma2.1} for $k=J+1$ and $\ell=i$, we have
\begin{align*}
\mathrm{HP}_{i}^{J+1}&= \sum_{j=1}^{i}q^{(J+1)(j-1)}\mathrm{HP}_{r-j+1}^{J+2}\\
&=\sum_{j=1}^{r} B_{i,j,(r-1)(J+1)+j}^{J}\mathrm{HP}_{r-j+1} ^{J+2}.
\end{align*}
Hence, \eqref{t1} is true for $d=J+1$. Next, we assume that \eqref{t1} is true for all $J+1 \leq d \leq s$. Then, by the induction hypothesis for $d=s$ we have
$$\mathrm{HP}_{i}^{J+1}=\sum_{j=1}^{r} B_{i,j,(r-1)s+j}^{J}\mathrm{HP}_{r-j+1} ^{s+1}.$$
Now, we prove \eqref{t1} for $d=s+1$. Replacing $\mathrm{HP}_{r-j+1} ^{s+1}$ in the above equation with the aid of Lemma \ref{Lemma2.1}, we obtain
\begin{align*}
\mathrm{HP}_{i}^{J+1}= \sum_{j=1}^{r} B_{i,j,(r-1)s+j}^{J} \left( \sum_{m=1}^{r-j+1}q^{(s+1)(m-1)}\mathrm{HP}_{r-m+1}^{s+2} \right). 
\end{align*}
Rewriting the above equation yields
\begin{align*}
\mathrm{HP}_{i}^{J+1}=&\sum_{\ell=1}^{r} \left( q^{(s+1)(r-\ell)} \sum_{j=1}^{\ell} B_{i,j,(r-1)s+j}^{J}  \right)\mathrm{HP}_{\ell} ^{s+2} \\
=&\sum_{\ell=1}^{r} B_{i,r-\ell+1,(r-1)(s+1)+(r-\ell+1)}^{J} \mathrm{HP}_{\ell} ^{s+2} \\
=&\sum_{j=1}^{r} B_{i,j,(r-1)(s+1)+j}^{J}\mathrm{HP}_{r-j+1} ^{s+2}.
\end{align*}
Hence, \eqref{t1} is true for $d=s+1$. This completes the proof. 
\end{proof}
Next, we state a recursion formula for $\mathcal{A}_{(r-1)J+\ell}$, which is given by Coulson \textit{et al.} \cite{Coulson_2017}.
\begin{lemma}[\cite{Coulson_2017}]\label{Lemma2.3}
Let $J$ be a nonnegative integer, and let $r, \ell$ be integers with $r \geq 2,~1 \leq \ell \leq r$. Then, for every integer $d \geq J+1$,  we have the following recursion formula:
$$\mathcal{A}_{(r-1)J+\ell}=\sum_{j=1}^{r} A_{\ell,j,(r-1)d+j}^J \mathcal{A}_{(r-1)d+j}.$$
Here, the coefficients $A_{\ell,j,(r-1)d+j}^J \in \mathbb{F}[[q]]$ satisfy the following recursion formula for $1 \leq j \leq r$  
$$A_{\ell,j,(r-1)(d+1)+j}^J= q^{(d+1)(j-1)} \sum_{m=1}^{r-j+1} A_{\ell,m,(r-1)d+m}^J$$
with the following initial conditions
$$A_{\ell,j,(r-1)(J+1)+j}^J=\begin{cases}
q^{(J+1)(j-1)} &\text{ if } 1\leq j \leq r-\ell+1;  \\
0 &\text{ if } r-\ell+2 \leq j \leq r.
\end{cases}$$
\end{lemma}
We get the recursion formula for $\mathcal{A}_{(r-1)J+\ell}$ from \cite[p. 122]{Coulson_2017}, with $G$ replaced by $\mathcal{A}$.
The recursion formula for $A^J_{\ell,j,(r-1)d+j}$ is given in \cite[Proposition 8.4]{Coulson_2017}. Here, we represent $~ _i ^{J} h_{l}^{(j)}$ of \cite[p. 122]{Coulson_2017} by $A_{\ell,j,(r-1)d+j}^{J}$, where $j$ is replaced with $d$, $i$ is replaced with $\ell$, and $l$ is replaced with $j$.
\par To prove Theorem \ref{Thm2}, it is enough to prove that $\mathrm{HP}^{J+1}_{i}=\mathcal{A}_{(r-1)J+\ell}$ (see, \eqref{eqj2}), where $\ell=r-i+1$. In this regard, we first prove that the coefficients in the recursion formulae of $\mathrm{HP}^{J+1}_{i}$ (Lemma \ref{Lemma2.2}) and $\mathcal{A}_{(r-1)J+\ell}$ (Lemma \ref{Lemma2.3}) are equal.
\begin{lemma}\label{Lemma2.4}
For all $d \geq J+1$, $r \geq 2$, and $1 \leq v \leq r$, we have 
\begin{align}\label{neq1}
A_{\ell,v,(r-1)d+v}^J=B_{i,v,(r-1)d+v}^J,
\end{align}
where $\ell=r-i+1$.
\end{lemma}
\begin{proof}
We prove \eqref{neq1} by induction on $d$. From Lemma \ref{Lemma2.3}, for $d=J+1$, we have
$$A_{\ell,v,(r-1)(J+1)+v}^{J}=\begin{cases}
	q^{(J+1)(v-1)} &\text{ if } 1\leq v \leq r-\ell+1;  \\
	0 &\text{ if } r-\ell+2 \leq v \leq r.
\end{cases}$$
We put $\ell=r-i+1$ in the above equation to obtain the following:
$$A_{r-i+1,v,(r-1)(J+1)+v}^{J}=\begin{cases}
	q^{(J+1)(v-1)} &\text{ if } 1\leq v \leq i;  \\
	0 &\text{ if } i+1 \leq v \leq r.
\end{cases}$$
Clearly, the expression above and Lemma \ref{Lemma2.2} yield
\begin{align*}
A_{\ell,v,(r-1)(J+1)+v}^{J}=B_{i,v,(r-1)(J+1)+v}^{J},
\end{align*}
for $\ell=r-i+1$. This proves \eqref{neq1} for $d=J+1$.
\par Now, we assume that \eqref{neq1} is true for all $J+1 \leq d\leq s$. By the induction hypothesis for $d=s$, and $\ell=r-i+1$, we have
\begin{align}\label{neq2}
A_{\ell,v,(r-1)s+v}^{J}=B_{i,v,(r-1)s+v}^{J}. 
\end{align}
By Lemma \ref{Lemma2.3}, we have 
$$A_{\ell,v,(r-1)(s+1)+v}^{J}= q^{(s+1)(v-1)} \sum_{m=1}^{r-v+1} A_{\ell,m,(r-1)s+m}^{J}.$$
Using \eqref{neq2} in the above equation, we obtain
\begin{align*}
A_{\ell,v,(r-1)(s+1)+v}^{J}&= q^{(s+1)(v-1)} \sum_{m=1}^{r-v+1} B_{i,m,(r-1)s+m}^{J}\\
&=B_{i,v,(r-1)(s+1)+v}^{J}\hspace{1cm} (\text{due to Lemma \ref{Lemma2.2}}).
\end{align*}
Thus, \eqref{neq1} is true for $d=s+1$. This proves \eqref{neq1} for all $d\geq J+1$.	
\end{proof}
Now, we are ready to prove Theorem \ref{Thm2}. In the following theorem, we prove that $\mathcal{A}_{(r-1)J+\ell}=\mathrm{HP}^{J+1}_{i}$, which proves Theorem \ref{Thm2}.
\begin{theorem}\label{Theorem2.5}
For $r \geq 2$, $1 \leq i \leq r$, and $J\geq0$, we have 
$$\mathcal{A}_{(r-1)J+\ell}=\mathrm{HP}_i^{J+1},$$ 
where $\ell=r-i+1$.
\end{theorem}
\begin{proof}
From Lemma \ref{Lemma2.3}, we have
$$A_{\ell,m,(r-1)(d+1)+m}^J= q^{(d+1)(m-1)} \sum_{t=1}^{r-m+1} A_{\ell,t,(r-1)d+t}^J.$$
Clearly, the exponent of $q$ in $A_{\ell,m,(r-1)(d+1)+m}^{J}$ is at least $(d+1)(m-1)$. Therefore, for all $1\leq m \leq r$, $  \lim_{d \to +\infty}  A_{\ell,m,(r-1)(d+1)+m}^{J} $ exists. Here, the limits are taken in $q$-adic topology, as discussed in Section \ref{Section2}. In particular, for $2\leq m \leq r$, we have
\begin{align*}
\lim_{d \to +\infty}A_{\ell,m,(r-1)(d+1)+m}^{J}= \lim_{d \to +\infty} \left(  q^{(d+1)(m-1)} \left(  \sum_{t=1}^{r-m+1} A_{\ell,t,(r-1)d+t}^J    \right)  \right)=0.
\end{align*}
By Lemma \ref{Lemma2.3}, we have 
$$\mathcal{A}_{(r-1)J+\ell}=\sum_{j=1}^{r} A_{\ell,j,(r-1)(d+1)+j}^{J} \mathcal{A}_{(r-1)(d+1)+j}.$$
By taking the limit $d \to +\infty$ on both sides of the above equation, we get
\begin{align}\label{g1}
\mathcal{A}_{(r-1)J+\ell}= \lim_{d \to +\infty} A_{\ell,1,(r-1)(d+1)+1}^{J} \mathcal{A}_{(r-1)(d+1)+1}.
\end{align}
 From \cite[Remark 8.1]{Coulson_2017}, we find that
\begin{align}\label{g2}
 \lim_{d \to +\infty} \mathcal{A}_{(r-1)(d+1)+1}=1.  
\end{align}
Hence, \eqref{g1} and \eqref{g2} yield 
$$\mathcal{A}_{(r-1)J+\ell}=\lim_{d \to +\infty} A_{\ell,1,(r-1)(d+1)+1}^{J}. $$
As mentioned earlier that $\lim_{d \to +\infty} A_{\ell,1,(r-1)(d+1)+1}^{J}$ exists, let us denote it by $A^J_{\ell,1,\infty}$. Hence,
\begin{align}\label{g3}
\mathcal{A}_{(r-1)J+\ell}=A_{\ell,1,\infty}^{J}.
\end{align}
\par Next, we consider the coefficients in the recursion formula of $\mathrm{HP}^{J+1}_{i}$ under a similar analysis. By Lemma \ref{Lemma2.2}, we have
$$B_{i,m,(r-1)(d+1)+m}^{J}= q^{(d+1)(m-1)} \sum_{t=1}^{r-m+1} B_{i,t,(r-1)d+t}^{J}.$$
Similarly, for all $1\leq m \leq r$, $  \lim_{d \to +\infty}  B_{i,m,(r-1)(d+1)+m}^{J} $ exists, and for $2\leq m \leq r$, $\lim_{d \to +\infty}  B_{i,m,(r-1)(d+1)+m}^{J}=0$. From Lemma \ref{Lemma2.2}, for $d \geq J+1$, we have
$$\mathrm{HP}_i ^{J+1}= \sum_{j=1}^{r} B_{i,j,(r-1)(d+1)+j}^{J}\mathrm{HP}_{r-j+1} ^{d+2}.$$
Taking the limit $d \to +\infty$ on both sides of the above equation and using note (N2), we obtain
\begin{align}\label{g4}
\mathrm{HP}_i ^{J+1}&= \lim_{d \to +\infty} B_{i,1,(r-1)(d+1)+1}^{J}\mathrm{HP}_{r} ^{d+2}\nonumber\\
&= \lim_{d \to +\infty} B_{i,1,(r-1)(d+1)+1}^{J} \mathrm{HP}^{d+2}.
\end{align}
Now, we find the value of $ \lim_{d \to +\infty}\mathrm{HP}^{d+2}=1$. We note that $$\mathrm{HP} ^{d+2}=\mathrm{HP} \left( \frac{S_{d+2}}{P_{d+2}}\right).$$ 
In the graded algebra $\frac{S_{d+2}}{P_{d+2}}$, the zeroth homogeneous component $A_0$ is isomorphic to $\mathbb{F}$. Therefore, the zeroth component has dimension $1$. For $1 \leq u \leq d+1$, the homogeneous component $A_u$ is zero as there are no monomials of weight between $1$ and $d+1$. Therefore, there exists some $f(q) \in \mathbb{F}[[q]]$ such that $\mathrm{HP}^{d+2}=1+q^{d+2}f(q)$. This implies that
\begin{align}\label{g5}
\lim_{d \to +\infty}\mathrm{HP}^{d+2}=1.
\end{align}
From \eqref{g4} and \eqref{g5}, we obtain
$$\mathrm{HP}_i ^{J+1}=\lim_{d \to +\infty} B_{i,1,(r-1)(d+1)+1}^{J}.$$
As mentioned earlier, $\lim_{d \to +\infty} B_{i,1,(r-1)(d+1)+1}^{J}$ exists, let us denote it by $B_{i,1,\infty}^{J}$. Hence,
\begin{align}\label{g6}
\mathrm{HP}_i ^{J+1}=B_{i,1,\infty}^{J}.
\end{align}
By Lemma \ref{Lemma2.4}, for $1\leq m \leq r$, we have
\begin{align*}
A_{\ell,m,(r-1)(d+1)+m}^{J}=B_{i,m,(r-1)(d+1)+m}^{J},
\end{align*}
where $\ell=r-i+1$. Taking the limit $d \to +\infty$ on both sides of the above equation yields
\begin{align}\label{g7}
A_{\ell,1,\infty}^{J}=B_{i,1,\infty}^{J}.
\end{align}
From \eqref{g3}, \eqref{g6}, and \eqref{g7}, we have
$$\mathcal{A}_{(r-1)J+\ell}=A_{\ell,1,\infty}^{J}=B_{i,1,\infty}^{J}=\mathrm{HP}_i ^{J+1},$$
where $\ell=r-i+1$.	This completes the proof.
\end{proof}
\section{concluding remark}
In this article, we study $J$-generalization of the Rogers-Ramanujan-Gordon identities using methods from commutative algebra. There are numerous other significant generalizations of the Rogers-Ramanujan identities in the literature. It would be worthwhile to investigate whether the approach used here can be extended to other Rogers-Ramanujan-type identities.

\section{Data Availability Statements}
Data sharing not applicable to this article as no datasets were generated or analysed during the current study.
\section{Conflict of interest}
The authors assert that there are no conflicts of interest.


\begin{thebibliography}{999}
\bibitem{Afsharijoo_2021}
P. Afsharijoo, {\it Looking for a new version of Gordon's identities}, Ann. Comb. 25 (2021), 543--571.

\bibitem{Afsharijoo_2026}
P. Afsharijoo, P. D. Gonz\'{a}lez P\'{e}rez, and H. Mourtada, {\it Partition identities associated with $A_r$-Surface singularities}, arXiv preprint, arXiv:2601.12048, 2026.

\bibitem{Andrews_1998}
G. E. Andrews, {\it The Theory of Partitions}, Cambridge Mathematical Library, Cambridge University Press, Cambridge, 1998. Reprint of the 1976 original.	

\bibitem{Andrews_1974}
G. E. Andrews, { \it An analytic generalization of the Rogers-Ramanujan identities for odd moduli}, Proc. Natl. Acad. Sci. USA 71 (1974), 4082--4085.

\bibitem{Andrews-Baxter_1984}
G. E. Andrews, R. J. Baxter, and P. J. Forrester, {\it Eight-vertex SOS model and generalized Rogers-Ramanujan-type identities}, J. Statist. Phys. 35 (1984), 193--266.

\bibitem{Atiyah_book}
M. F. Atiyah and I. G. Macdonald, {\it Introduction to Commutative Algebra}, Addison-Wesley Publishing Co., Reading, Mass.-London-Don Mills, Ont., 1969.

\bibitem{Baxter_1981}
R. J. Baxter, {\it Rogers-Ramanujan identities in the hard hexagon model}, J. Statist. Phys. 26 (1981), 427--452.

\bibitem{Ono_2008a}
K. Bringmann, K. Ono, and R. Rhoades, {\it Eulerian series as modular forms}, J. Amer. Math. Soc. 21 (2008), 1085--1104.

\bibitem{Mourtada_2013}
C. Bruschek, H. Mourtada, and J. Schepers, {\it Arc spaces and the Rogers-Ramanujan identities}, Ramanujan J. 30 (2013), 9--38.

\bibitem{Coulson_2017}
B. Coulson, S. Kanade, J. Lepowsky, R. McRae, F. Qi, M. C. Russell, and C. Sadowski, {\it A motivated proof of the G\"ollnitz-Gordon-Andrews identities}, Ramanujan J. 42 (2017), 97--129.

\bibitem{Eisenbud}
D. Eisenbud, {\it Commutative Algebra with a View toward Algebraic Geometry}, Graduate Texts in Mathematics, Vol. 150, Springer-Verlag, 1995.

\bibitem{Euler_1748}
L. Euler, {\it Introductio in analysin infinitorum}, Marcum-Michaelem Bousquet, Lausannae, 1748.

\bibitem{Fulman_2000}
J. Fulman, {\it The Rogers-Ramanujan identities, the finite general linear groups, and the Hall-Littlewood polynomials}, Proc. Amer. Math. Soc. 128 (2000), 17--25.
          
\bibitem{Gordon_1961}
B. Gordon, {\it A combinatorial generalization of the Rogers-Ramanujan identities}, Amer. J. Math. 83 (1961), 393--399. 
       
\bibitem{Pfister_book}
G.-M. Greulel and G. Pfister, {\it A Singular Introduction to Commutative Algebra}, Springer-Verlag, Berlin, 2002. With contributions by Olaf Bachmann, Christoph Lossen and Hans Schonemann, With 1 CD-ROM (Windows, Macintosh, and UNIX).

\bibitem{Ono}
M. J. Griffin, K. Ono, and S. O. Warnaar, {\it A framework of Rogers-Ramanujan identities and their arithmetic properties}, Duke Math. J. 165 (2016), no. 8, 1475--1527.

\bibitem{Lepowsky and Zhu 2012} 
J. Lepowsky and M. Zhu, \textit{A motivated proof of Gordon's identities}, Ramanujan J. 29 (2012), 199--211.

\bibitem{MacMahon}
P. A. MacMahon, {\it Combinatory Analysis}, Vol. I and II, Chelsea Publ., New York, 1960.

\bibitem{Mourtada_2025}
H. Mourtada, {\it Hilbert meets Ramanujan: singularity theory and integer partitions}, Bull. Amer. Math. Soc. 62 (2025), no. 1, 93--111.

\bibitem{Ono_2008b}
K. Ono, {\it Unearthing the visions of a master: harmonic Maass forms and number theory}, in Current Developments in Mathematics, 2008, pp. 347--454, Int. Press, Somerville, MA, 2009.

\bibitem{Richmond_1981}
B. Richmond and G. Szekeres, {\it Some formulas related to dilogarithms, the zeta function and the Andrews-Gordon identities}, J. Austral. Math. Soc. Ser. A 31 (1981), 362--373.

\bibitem{Rogers_1894}
L. J. Rogers, {\it Second Memoir on the Expansion of certain Infinite Products}, Proceedings of the London Mathematical Society 25 (1893/94), 318--343.

\bibitem{Schur}
I. Schur, {\it Ein Beitrag zur additiven Zahlentheorie und zur Theorie der Kettenbr\"{u}che}, S.-B. Preuss. Akad. Wiss. Phys. Math. Klasse (1917), 302–321.

\end{thebibliography}
\end{document}